\documentclass[11pt,reqno]{amsart}

\usepackage[dvipsnames]{xcolor}
\usepackage{amsmath, amsthm, amssymb, amsfonts, listings, hyperref, multicol, bm, shuffle, xcolor, enumerate, dsfont,tikz, tabu, Frobenius-MACROS,graphicx,mathtools,array,todonotes,bbm}
\usepackage{mathabx}
\usepackage{enumitem}
\usepackage{enumitem}

\usepackage{tikz}
\usetikzlibrary{decorations.pathreplacing}
\usepackage{tikz}

\theoremstyle{definition}
\newtheorem{theorem}{Theorem}[section]
\newtheorem{corollary}[theorem]{Corollary}

\newtheorem{lemma}[theorem]{Lemma}
\newtheorem{remark}[theorem]{Remark}
\newtheorem{proposition}[theorem]{Proposition}
\newtheorem{definition}[theorem]{Definition}
\newtheorem{example}[theorem]{Example}

\usetikzlibrary{decorations.pathreplacing,shapes}
\usepackage{tikz}

\definecolor{lightblue}{RGB}{240,240,255}
\definecolor{lightred}{RGB}{255,240,240}

\newcommand{\omitt}[1]{}

\title[Coherent systems for free Frobenius towers]{Coherent systems of probability measures on graphs for representations of free Frobenius towers}

\date{}

\author{Henry Kvinge}

\begin{document}

\maketitle

\begin{abstract}
First formally defined by Borodin and Olshanski, a coherent system on a graded graph is a sequence of probability measures which respect the action of certain down/up transition functions between graded components. In one common example of such a construction, each measure is the Plancherel measure for the symmetric group $\Sy{n}$ and the down transition function is induced from the inclusions $\Sy{n} \hookrightarrow \Sy{n+1}$. 

In this paper we generalize the above framework to the case where $\{A_n\}_{n \geq 0}$ is any free Frobenius tower and $A_n$ is no longer assumed to be semisimple. In particular, we describe two coherent systems on graded graphs defined by the representation theory of $\{A_n\}_{n \geq 0}$ and connect one of these systems to a family of central elements of $\{A_n\}_{n \geq 0}$. When the algebras $\{A_n\}_{n \geq 0}$ are not semisimple, the resulting coherent systems reflect the duality between simple $A_n$-modules and indecomposable projective $A_n$-modules.

\end{abstract}

\tableofcontents


\section{Introduction}

Study of the asymptotic representation theory of symmetric groups has uncovered deep results through the synthesis of techniques from probability theory, combinatorics, and algebra. One idea frequently used both implicitly and explicitly in this field is that of a coherent system. Coherent systems were first formally defined in \cite{BO09} where they are a key part of the framework used to construct infinite-diffusion processes. Given a graded graph with vertex set $V = \bigcup_{n \geq 0}V_n$, a coherent system consists of a sequence of probability measures $\{M_n\}_{n \geq 0}$ on the graded components of $V$, which are consistent with the action of a down transition function $p_\downarrow$ between components.

A basic example of this construction occurs when $Pl_n$ is taken to be the Plancherel measure on $\MB{Y}_n$. The sequence $\{Pl_n\}_{n \geq 0}$ is coherent with respect to down and up transition functions each controlled by the number of standard Young tableau of appropriate shape. On the other hand, this coherent system can also be described in terms of the representation theory of the tower of symmetric group algebras $\{\MB{C}[\Sy{n}]\}_{n \geq 0}$. From this point of view, $Pl_n$ is defined via the decomposition of the left regular representation of $\MB{C}[\Sy{n}]$ into a direct sum of simple $\MB{C}[\Sy{n}]$ representations, and $p_\downarrow$ and $p_\uparrow$ are defined via induction and restriction functors. 

One of the main goals of this paper is to understand what assumptions are necessary to generalize the above construction to an arbitrary tower of algebras $\{A_n\}_{n \geq 0}$. We show that when for all $n \geq 0$: 
\vspace{-.5mm}
\begin{enumerate}\itemsep-.2em 
\item the induction and restriction functors between $A_n$ representations and $A_{n+1}$ representations are biadjoint (so that $A_{n+1}$ is a Frobenius extension of $A_n$), 
\item $A_{n+1}$ is a free $(A_n,A_n)$-bimodule, 
\end{enumerate}
then there a two coherent systems that generalize the Plancherel system corresponding to $\{\MB{C}[\Sy{n}]\}_{n \geq 0}$. The observation that there are two such systems to choose from reflects the fact that when $A_n$ is not semisimple, there arises a distinction between simple $A_n$-modules and indecomposable projective $A_n$-modules. In particular, in one coherent system the down transition functions are controlled by the dimensions of simple $A_n$-modules and the up transition functions are controlled by the dimensions of indecomposable projective $A_n$-modules, while in the other coherent system the reverse is true.

That the existence of coherent systems should depend on categorical conditions is not surprising given the recent appearance of coherent systems within categorical representation theory. It was shown in \cite{KLM16} for example, that certain moments associated to down/up transition functions for the tower $\{\MB{C}[\Sy{n}]\}_{n \geq 0}$ appear naturally within the center of Khovanov's Heisenberg category \cite{Kho14}, a monoidal category which conjecturally categorifies the infinite dimensional Heisenberg algebra. A similar phenomenon \cite{KOR17} was observed for down/up transition functions associated to towers of Sergeev algebras, which appear in the twisted Heisenberg category of Cautis and Sussan \cite{CS15}. 

In the classical case of $\{\MB{C}[\Sy{n}]\}_{n \geq 0}$, Biane observed that data related to $p_\downarrow$ is encoded by elements of $\{Z(\MB{C}[\Sy{n}])\}_{n \geq 0}$ \cite{B98}. This allows for certain questions related to asymptotic processes to be translated into algebraic/combinatorial questions. We show that when each $A_n$ in the free Frobenius tower $\{A_n\}_{n \geq 0}$ is further assumed to be a symmetric Frobenius algebra, then there are families of elements in $\{Z(A_n)\}_{n \geq 0}$ which encode data not only for $p_\downarrow$, but also $p_\uparrow$.  
We plan to investigate these elements and their possible connection to symmetric functions in greater detail in a later paper. 

One of the broader goals of this work is to pave the way for studies of the asymptotic representation theory of more exotic towers of algebras. In particular, because of their connections to Lie theory and geometry, we think that understanding the asymptotic representation theory of the wreath product algebras attached to a Frobenius algebra $F$ would be particularly interesting (see Example \ref{example-towers-of-algebras}.\ref{ex-wreath}). Degenerate cyclotomic Hecke algebras are another natural generalization of symmetric groups (see Example \ref{example-towers-of-algebras}.\ref{ex-cyclo}). This family of non-semisimple towers of algebras has deep connections to Lie theory \cite[Part 1]{K05}. Furthermore, as their representation theory is indexed by certain multipartitions, we hypothesize that using the coherent systems described in this paper, an entire new family of limit shapes could be constructed.

The paper is structured as follows: in Section \ref{sect-Frobenius-extensions} we review Frobenius extensions and Frobenius towers. In Section \ref{sect-rep-theory} we describe the representation theory of Frobenius towers particularly in relation to elements from the centralizers $\{Z(A_{n+1},A_n)\}_{n \geq 0}$. In Section \ref{sect-coherent-measures} we define a pair of coherent systems on certain graphs attached to free Frobenius towers and connect one of these to a family of central elements. Finally in Section \ref{sect-future-direct} we discuss future directions of research.

{\bf{Acknowledgements:}} The idea for this paper arose from extended discussions with Anthony Licata. The author would like to thank him for these and for continued feedback during the course of this paper's development. Without his help, it is unlikely that this paper would have been written. The author would also like to thank Alistair Savage, Monica Vazirani, Leonid Petrov, and Ben Webster for helpful discussions. 


\section{Frobenius extensions and Frobenius towers} \label{sect-Frobenius-extensions}

In this section we review some of the basic properties of Frobenius extensions and Frobenius towers. All the algebras that we introduce are assumed to be finite-dimensional $\MB{C}$-algebras unless otherwise stated. We point the reader toward \cite{Kad99} for a detailed introduction to Frobenius extensions as well as a more diverse set of examples. 

Recall that a tower of $\MB{C}$-algebras $\{A_n\}_{n \geq 0}$ is a nested sequence of algebras
\begin{equation} \label{eqn-tower-def}
A_0 \subset A_1 \subset A_2 \subset \dots
\end{equation}
In this paper we will always assume that $A_0 = \MB{C}$ which adds nice properties to the tower. Due to the inclusions in \eqref{eqn-tower-def}, for any $0 \leq k_1,k_2 < n$, $A_n$ is an $(A_{k_1},A_{k_2})$-bimodule, with the action of $A_{k_1}$ given by the natural left-multiplication against $A_n$ and the action of $A_{k_2}$ given by the natural right-multiplication against $A_n$. When considering $A_n$ as a bimodule we will sometimes use the following notation:
\begin{itemize}
\item To denote $A_n$ as a $(A_{k_1},A_{k_2})$-bimodule we write $_{k_1}(A_n)_{k_2}$.
\item To denote $A_n$ as a $(A_{k_1},A_{n})$-bimodule we write $_{k_1}(A_n)$.
\item To denote $A_n$ as a $(A_{n},A_{k_2})$-bimodule we write $(A_n)_{k_2}$.
\item To denote $A_n$ as a $(A_n,A_n)$-bimodule we write $A_n$.
\end{itemize}
This bimodule structure defines a restriction functor $\res^{A_{n}}_{A_k}: A_{n}\Mod \rightarrow A_k\Mod$ and an induction functor $\ind^{A_{n}}_{A_k}: A_k \Mod \rightarrow A_{n} \Mod$ for each $0 \leq k < n$, where $A_k\Mod$ (resp. $A_n\Mod$) is the category of finite-dimensional $A_k$-modules (resp. $A_n$-modules). Specifically for $N \in A_n\Mod$ and $K \in A_k\Mod$,
\begin{equation} \label{eqn-def-res}
N \xrightarrow{\res^{A_{n}}_{A_k}} \;_{k}(A_n) \otimes_{A_n} N. 
\end{equation}
and
\begin{equation} \label{eqn-def-ind}
K \xrightarrow{\ind^{A_{n}}_{A_k}} (A_n)_{k} \otimes_{A_k} K.
\end{equation}
Henceforth for brevity we write
\begin{equation*}
\ind^{n}_k := \ind^{A_{n}}_{A_k} \quad\quad \text{and} \quad\quad \res^{n}_{k} := \res^{A_{n}}_{A_k}
\end{equation*}
when the tower of algebras is understood from the context. Via \eqref{eqn-def-res} we can identify $\res^{n}_k$ with the $(A_k,A_n)$-bimodule $_k(A_n)$ and via \eqref{eqn-def-ind} we can identify $\ind^{n}_k$ with the $(A_n,A_k)$-bimodule $(A_n)_k$. This description allows us to translate natural transformations between compositions of the functors $\res^{n}_k$ and $\ind^{n}_k$ into bimodule homomorphisms between tensor products of $_k(A_n)$ and $(A_n)_k$.

\begin{example} \label{example-towers-of-algebras}
\begin{enumerate}
\item \label{ex-sym-grp} The tower of symmetric group algebras $\{\MB{C}[\Sy{n}]\}_{n \geq 0}$ is our motivating example for this paper. Recall that $\Sy{n}$ is generated by Coxeter generators $s_1, s_2, \dots, s_{n-1}$ subject to relations
 \begin{align} \label{eqn-sym-grp-relations1}
& s_i^2 = 1,   & 1 \leq i \leq n-1,\\
 &s_is_j = s_js_i, & |i-j| > 1, \;\; 1 \leq i,j \leq n-1,\\
 & s_is_{i+1}s_i = s_{i+1}s_is_{i+1}, & 1 \leq i \leq n-2. \label{eqn-sym-grp-relations3}
 \end{align}
We will always assume that the inclusion $\Sy{k} \hookrightarrow \Sy{n}$ is the standard one with $\Sy{k}$ mapping to the subgroup of $\Sy{n}$ generated by $s_1, s_2, \dots, s_{k-1}$. By linear extension this defines an inclusion $\MB{C}[\Sy{k}] \hookrightarrow \MB{C}[\Sy{n}]$. 

\item \label{ex-sergeev} Recall that the Clifford superalgebra $Cl_n$ is the unital associative algebra with $n$ generators $c_1, c_2, \dots, c_n$ such that for $1 \leq i,j \leq n$: 
\begin{equation*} 
 c_i^2=1 \quad\quad \text{and} \quad\quad c_ic_j=-c_jc_i \text{ for } i\neq j.
\end{equation*}
The superalgebra structure is defined by setting each generator $c_i$ to be an odd element. The {\emph{Sergeev superalgebra}} is defined as $\MB{S}_n := Cl_n \rtimes \mathbb{C}[\Sy{n}]$, where the action of $\Sy{n}$ is given by permuting the indices of $c_1, c_2, \dots, c_n$ so that:
\begin{align*}
&s_ic_i = c_{i+1}s_i, & 1 \leq i \leq n-1, \\
&s_ic_{i+1} = c_is_i, & 1 \leq i \leq n-1,\\
&s_ic_j = c_js_i  & j \neq i,\;i+1.
\end{align*}
The Sergeev algebras form a tower $\{\MB{S}_n\}_{n \geq 0}$ via the embedding 
\begin{equation*}
\MB{S}_n \hookrightarrow \MB{S}_{n+1}
\end{equation*}
which sends $c_i \mapsto c_i$ for $1 \leq i \leq n$ and $s_i \mapsto s_i$ for $1 \leq i \leq n-1$.

One reason that Sergeev algebras are interesting is that studying simple representations of $\{\MB{S}_n\}_{n \geq 0}$ is equivalent to studying projective representations of $\{\Sy{n}\}_{n \geq 0}$. See \cite{K05} and \cite{WW11} for detailed studies of the algebras $\{\MB{S}_n\}_{n \geq 0}$, their representation theory, and connections to combinatorics and Lie theory. $\{\MB{S}\}_{n \geq 0}$ is another tower of algebras whose asymptotic representation theory has been well-studied, see for example \cite{Iv01}, \cite{Iv06}, \cite{Bor97}, \cite{Pet09}. 

\item \label{ex-wreath} Let $F$ be a Frobenius graded superalgebra (that is, $F$ is $\MB{Z} \times \MB{Z}_2$ graded). The symmetric group $\Sy{n}$ acts on $F^{\otimes n}$ by superpermutations. More precisely, for homogeneous $f_1, f_2, \dots, f_n \in F$ and $1 \leq i \leq n-1$,
\begin{equation} \label{eqn-wreath-prod-action}
s_i \cdot (f_1 \otimes \dots \otimes f_i \otimes f_{i+1} \otimes \dots \otimes f_n) 
\end{equation}
\begin{equation*}
= (-1)^{\bar{f_i}\widebar{f_{i+1}}}(f_1 \otimes \dots \otimes f_{i+1} \otimes f_{i} \otimes \dots \otimes f_n)
\end{equation*}
where $\bar{f_i}$ and $\widebar{f_{i+1}}$ denote the $\mathbb{Z}_2$-degree of $f_i$ and $f_{i+1}$ respectively. 

The wreath product algebras $\{F^{\otimes n} \rtimes \Sy{n} \}_{n \geq 0}$ induced from \eqref{eqn-wreath-prod-action} form a tower via the embedding
\begin{equation*}
F^{\otimes n} \rtimes \Sy{n} \hookrightarrow F^{\otimes (n+1)} \rtimes \Sy{n+1}
\end{equation*}
which for $f_1, \dots, f_n \in F$ and $\omega \in \Sy{n}$ sends 
\begin{equation*}
(f_1 \otimes \dots \otimes f_n)\omega \mapsto ((f_1 \otimes \dots \otimes f_n \otimes 1)\omega.
\end{equation*}
$F^{\otimes n} \rtimes \Sy{n}$ also inherits a $(\MB{Z} \times \MB{Z}_2)$-grading from $F$ by setting $\Sy{n}$ to sit in degree $(0,0)$.

In fact, Example \ref{example-towers-of-algebras}.\ref{ex-sergeev} is a special case of $F^{\otimes n} \rtimes \Sy{n}$ with $F = Cl_1$. For a description of the representation theory of $F^{\otimes n} \rtimes \Sy{n}$ in the non-super, ungraded setting see \cite[Theorem A.5]{Mac15}, while for the super, graded case see \cite[Section 4]{RS17}.

\item \label{ex-cyclo} The {\emph{degenerate affine Hecke algebra}} $H_n$ is generated by elements $s_1, s_2, \dots s_{n-1}$ and $x_1, x_2, \dots, x_n$, such that $s_1, \dots, s_{n-1}$ satisfy relations \eqref{eqn-sym-grp-relations1}-\eqref{eqn-sym-grp-relations3}, $x_1, \dots, x_n$ commute, and:
\begin{align*}
& s_jx_i = x_is_j, & i \neq j, j+1,\\
& s_ix_i = x_{i+1}s_i - 1, & 1 \leq i \leq n-1.
\end{align*}
Thus as a $\MB{C}$-vector space 
\begin{equation*}
H_n \cong  \MB{C}[x_1,\dots,x_n] \otimes \MB{C}[\Sy{n}].
\end{equation*}
Let $P^+$ be the dominant weight lattice for $\mathfrak{g} = \mathfrak{sl}_\infty$ and let $I$ be the associated Dynkin indexing set. Also let $\lambda = \sum_{i \in I} \lambda_i \omega_i \in P^+$ where $\lambda_i \in \MB{Z}_{\geq 0}$ and $\omega_i$ is the $i$th fundamental weight. The integer $d = \sum_{i\in I} \lambda_i$ is called the {\emph{level}} of $\lambda$. Let $I^\lambda$ be the two-sided ideal of $H_n^\lambda$ generated by the element
\begin{equation*}
\prod_{i \in I} (x_1 - i)^{\lambda_i}.
\end{equation*}
The quotient algebra $H^\lambda_n = H_n/I^\lambda$ is called the {\emph{degenerate cyclotomic Hecke algebra}} associated to $\lambda$. By abuse of notation we write $x_i, s_i \in H_n^\lambda$ for the images of $x_i, s_i \in H_n$ in this quotient. $H^\lambda_n$ is finite dimensional, with the set
\begin{equation*}
\{ x_1^{t_1}x_2^{t_2} \dots x_n^{t_n}\sigma \;|\; t_1,t_2, \dots, t_n < d, \; \sigma \in \Sy{n}\}
\end{equation*}
being a basis \cite[Theorem 3.2.2]{K05}, so that $\dim(H^\lambda_n) = d^nn!$. The algebras $\{H_n^\lambda\}_{n \geq 0}$ generalize symmetric group algebras because when $\lambda = \omega_0$, $H^\lambda_n \cong \MB{C}[\Sy{n}]$.

The collection $\{H^\lambda_n\}_{n \geq 0}$ forms a tower via the embeddings
\begin{equation*}
H_{n}^\lambda \hookrightarrow H_{n+1}^\lambda
\end{equation*}
which send $x_i \mapsto x_i$ for $1 \leq i \leq n$ and $s_i \mapsto s_i$ for $1 \leq i \leq n-1$. The tower $\{H^\lambda_n\}_{n \geq 0}$ has a rich representation theory which is described in detail in \cite[Part I]{K05}. In particular it has a crystal structure isomorphic to the highest weight crystal $B(\lambda)$ in affine type $A$. The algebras $H^\lambda_n$ will be one of our prime examples of a tower of non-semisimple algebras.

\omitt{
Note that one can generalize the tower of Hecke algebras from Example \ref{example-towers-of-algebras}.\ref{ex-Hecke-algebras} via cyclotomic Hecke algebras in a manner analogous to the way that symmetric group algebras were generalized via degenerate cyclotomic Hecke algebras.
} 

\end{enumerate}
In all the examples above, we follow the usual convention that $A_0 = \MB{C}$.
\end{example}

\begin{remark}
It is worth noting that all the examples in Example \ref{example-towers-of-algebras} are particular instances of cyclotomic wreath product algebras as defined in \cite{Sav18}. While we could have thus presented a single unified example, we felt it was more appropriate to partition our examples into digestible chunks as above for the benefit of the reader that may not be familiar with cyclotomic wreath product algebras.
\end{remark}

For $0 \leq k < n$ it is always the case that the functor $\ind^{n}_k$ is left adjoint to $\res^{n}_k$. This is known as {\emph{Frobenius reciprocity}}. When $\res^{n}_k$ is also left adjoint to $\ind^{n}_k$, so that $(\ind^n_k,\res^n_k)$ are a pair of biadjoint functors, then we say that $A_{n}$ is a {\emph{Frobenius extension}} of $A_k$. When in addition $_k(A_{n})_k$ is a free $(A_n,A_n)$-bimodule, then we say that $A_{n}$ is a {\emph{free Frobenius extension}} of $A_k$. If $A_{n+1}$ is a Frobenius extension of $A_n$ for all $n \geq 0$ in the tower $\{A_n\}_{n \geq 0}$, then we say that $\{A_n\}_{n \geq 0}$ is a {\emph{Frobenius tower}}. When $A_{n+1}$ is a free Frobenius extension of $A_n$ for all $n \geq 0$, then we say that $\{A_n\}_{n \geq 0}$ is a {\emph{free Frobenius tower}}. 

\begin{theorem}\cite[Corollary 1.2]{BF93}
The algebra $A_{n+1}$ is a free Frobenius extension of $A_n$ if and only if there is a $(A_n,A_n)$-bimodule homomorphism 
\begin{equation*}
E_{n+1,n}: \;_n(A_{n+1})_n \rightarrow A_n,
\end{equation*}
a finite set of elements $B_{n+1,n} \subset A_{n+1}$, and a vector space isomorphism $^\vee: A_{n+1} \rightarrow A_{n+1}$ such that for any $a \in A_{n+1}$,
\begin{equation} \label{eqn-def-frob-ex}
a = \sum_{b \in B_{n+1,n}} E_{n+1,n}(ab^\vee)b = \sum_{b \in B_{n+1,n}} b^\vee E_{n+1,n}(b a).
\end{equation}
\end{theorem}

We call $B_{n+1,n}$ and $B_{n+1,n}^\vee$ a {\emph{dual basis pair}}, $E_{n+1,n}$ a {\emph{Frobenius homomorphism}}, and the triple $(E_{n+1,n},B,B^\vee)$ a {\emph{Frobenius system}}. 

\begin{corollary}
When $A_{n+1}$ is a free Frobenius extension of $A_n$ then $B_{n+1,n}$ is a basis for $A_{n+1}$ as a left $A_n$-module, $B_{n+1,n}^\vee$ is a basis for $A_{n+1}$ as a right $A_n$-module, and 
\begin{equation*}
E_{n+1,n}(b(b')^\vee) = \delta_{b,b'} \quad\quad \text{for any $b, b' \in B_{n+1,n}$}.
\end{equation*}
\end{corollary}

\begin{example} \label{ex-frobenius-ex-structure}
All the towers listed in Example \ref{example-towers-of-algebras} are examples of free Frobenius towers. We describe their Frobenius homomorphisms and examples of dual bases below. 

\begin{enumerate}

\item \label{ex-frobenius-ex-structure-symm} For $A_{n+1} = \MB{C}[\Sy{n+1}]$ and $A_n = \MB{C}[\Sy{n}]$, $E_{n+1,n}: \MB{C}[\Sy{n+1}] \twoheadrightarrow \MB{C}[\Sy{n}]$ is defined on group elements so that for $g \in \Sy{n+1}$ 
\begin{equation*}
E_{n+1,n}(g) = \begin{cases}
g & \text{if $g \in \Sy{n}$}\\
0 & \text{otherwise.}
\end{cases}
\end{equation*}
Alternatively, since as an $(\MB{C}[\Sy{n}],\MB{C}[\Sy{n}])$-bimodule $\MB{C}[\Sy{n+1}]$ decomposes as
\begin{equation} \label{eqn-decomposition-of-Sn}
_n(\MB{C}[\Sy{n+1}])_n \cong \MB{C}[\Sy{n}] \oplus \Big(\MB{C}[\Sy{n}]\otimes_{\MB{C}[\Sy{n-1}]} \MB{C}[\Sy{n}]\Big),
\end{equation}
then $E_{n+1,n}$ can also be viewed as projection onto the first summand on the right side of \eqref{eqn-decomposition-of-Sn}.

One choice for $B_{n+1,n}$ is the set of right coset representatives of $\Sy{n}$ in $\Sy{n+1}$. Then $B_{n+1,n}^\vee$ is the set of left coset representatives of $\Sy{n}$ in $\Sy{n+1}$. In particular, one dual basis pair is
\begin{equation*}
B_{n+1,n} = \{ s_{n}s_{n-1}\dots s_i \;| \; 1 \leq i \leq n+1\}
\end{equation*}
and 
\begin{equation*}
B_{n+1,n}^\vee = \{ s_i \dots s_{n-1}s_{n}\;| \; 1 \leq i \leq n+1\}
\end{equation*}
(here, the convention is that when $i = n+1$, the corresponding element is the identity). Note that in this case $b \xmapsto{^\vee} b^\vee$ sends an element of $B_{n+1,n}$ to its inverse.

\item Let $F$ be a Frobenius graded superalgebra with trace map $\tr: F \rightarrow \MB{C}$ of degree $(-\delta,\tau) \in \MB{Z} \times \MB{Z}_2$. Let $B_F$ and $\widehat{B_F}$ be homogeneous dual bases of $F$ with respect to $\tr$.

If $\{A_{n}\}_{n \geq 0} = \{F^{\otimes (n)} \rtimes \Sy{n}\}_{n \geq 0}$, then from \cite[Proposition 3.4]{RS17}, as an $(A_n,A_n)$-bimodule $A_{n+1}$ decomposes as
\begin{equation*}
_n(A_{n+1})_n \cong \Big( \bigoplus_{b \in B_F} (1 \otimes \dots \otimes 1 \otimes b)A_n\Big) \oplus \Big(A_n \otimes_{A_{n-1}} A_{n}\Big)
\end{equation*}
and
\begin{equation*}
(1 \otimes \dots \otimes 1 \otimes b)A_n \cong \Pi^{\bar{b}} A_n,
\end{equation*}
where $\Pi: A_n\Mod \rightarrow A_n\Mod$ is the parity shift functor (recall that $\bar{b} \in \MB{Z}_2$ is the parity of $b$ in the $\MB{Z}_2$-grading). Then $E_{n+1,n}: A_{n+1} \rightarrow A_n$ corresponds to the $(A_n,A_n)$-bimodule homomorphism that first projects
\begin{equation*}
A_{n+1} \twoheadrightarrow \bigoplus_{b \in B_F} (1 \otimes \dots \otimes 1 \otimes b)A_n
\end{equation*}
followed by an $(A_n,A_n)$-bimodule homomorphism
\begin{equation*}
\bigoplus_{b \in B_F} (1 \otimes \dots \otimes 1 \otimes b)A_n \twoheadrightarrow A_n.
\end{equation*}
From \cite[Proposition 3.4]{RS17}, for $f_1, \dots, f_{n+1} \in F$ and $\sigma \in \Sy{n+1}$,
\begin{align*}
&E_{n+1,n}\Big( (f_1 \otimes \dots \otimes f_{n+1})\sigma\Big)\\
&= \begin{cases}
(-1)^{\tau(\widebar{f_1}+ \dots + \widebar{f_n})}\tr(f_{n+1})(f_1 \otimes \dots \otimes f_n)\sigma & \text{if $\sigma \in \Sy{n}$}\\
0 & \text{otherwise}.
\end{cases}
\end{align*}
Corresponding dual bases are 
\begin{equation*}
B_{n+1,n} = \{ (1^{\otimes n} \otimes b)s_n \dots s_i \; | \; b \in B_F, \; 1 \leq i \leq n+1 \}
\end{equation*}
and
\begin{equation*}
B_{n+1,n}^\vee = \{ s_i \dots s_n(1^{\otimes n} \otimes b) \; | \; b \in \widehat{B_F}, \; 1 \leq i \leq n+1 \}.
\end{equation*}

\item For $\lambda \in P^+$ with $\lambda$ of level $d$, by \cite[Lemma 7.6.1]{K05}, as an $(H^\lambda_{n},H^\lambda_{n})$-bimodule, $H^\lambda_{n+1}$ decomposes as
\begin{equation*}
_n(H^\lambda_{n+1})_n \cong  \Big(\bigoplus^{d-1}_{j = 0} H_n^\lambda x_{n+1}^j\Big) \oplus (H^\lambda_n \otimes_{H^\lambda_{n-1}} H^\lambda_n)
\end{equation*}
where $H_n^\lambda x_{n+1}^j \cong H_n^\lambda$. Then $E_{n+1,n}: H^\lambda_{n+1} \twoheadrightarrow H^\lambda_{n}$ corresponds to projection to the summand $H^\lambda_{n}x_{n+1}^{d-1}$. Recall that in the case where $\lambda = \omega_0$ then $H_n^\lambda \cong \MB{C}[\Sy{n}]$, and $E_{n+1,n}$ is the Frobenius homomorphism described in Example \ref{ex-frobenius-ex-structure}.\ref{ex-frobenius-ex-structure-symm}. 

Define
\begin{equation*}
y_{n,k} := \sum_{t = k}^{d-1} (-1)^{d-1-t}x_n^{t-k}\det\Big(E_{n,n-1}(x^{d+j-i}_n)\Big)_{i,j = 1,\dots,d-1-t}
\end{equation*}
where the determinant in this expression is $1$ when $t = d-1$. By \cite[Proposition 5.11]{MS17}, a pair of dual bases with respect to $E_{n+1,n}$ is
\begin{equation*}
B_{n+1,n} = \{\;s_n \dots s_iy_{i,a} \;| \; i =1,\dots,n+1, \; a = 0,\dots,d-1 \;\}
\end{equation*}
and
\begin{equation*}
B_{n+1,n}^\vee = \{\;x^a_i s_i \dots s_n \;| \; i =1,\dots,n+1, \; a = 0,\dots,d-1 \;\}.
\end{equation*}

\end{enumerate}
\end{example}

While we only required that the algebra $A_{n+1}$ immediately above $A_n$ be a Frobenius extension of $A_n$, this assumption is in fact enough to ensure that for any $0 \leq k < n+1$, $A_{n+1}$ is a Frobenius extension of $A_k$.

\begin{proposition} \cite[Section 1.3]{Kad99} \label{prop-transitivity-of-extensions}
Let $\{A_n\}_{n \geq 0}$ be a Frobenius tower. Then for any $0 \leq k < n$, $A_n$ is a Frobenius extension of $A_k$ with Frobenius homomorphism $E_{n,k} := E_{k+1,k} \circ \dots \circ E_{n,n-1}$, and dual basis pair
\begin{equation*}
B_{n,k} := \{b^{(k+1)} \dots b^{(n-1)}b^{(n)}  \; | \; b^{(i)} \in B_{i,i-1}\;\}
\end{equation*}
and
\begin{equation*}
B_{n,k}^\vee := \{ (b^{(n)})^\vee(b^{(n-1)})^\vee \dots (b^{(k+1)})^\vee   \; | \; b^{(i)} \in B_{i,i-1}\;\}.
\end{equation*}
\end{proposition}

\begin{remark}
As a consequence of the assumption that $A_0 = \MB{C}$, each $A_n$ in the Frobenius tower $\{A_n\}_{n \geq 0}$ is actually a Frobenius algebra with trace $E_{n,0}: A_n \rightarrow \MB{C}$.
\end{remark}

When for all $a, a' \in A_n$, 
\begin{equation*}
E_{n+1,k}(aa') = E_{n+1,k}(a'a)
\end{equation*}
then we say that $A_{n+1}$ is a {\emph{symmetric}} Frobenius extension of $A_k$. Note that when this property is true for $E_{n,0}$, then $A_n$ is a {\emph{symmetric Frobenius algebra}} (also known as a {\emph{symmetric algebra}}). We will generally assume that each $A_n$ in our tower $\{A_n\}_{n \geq 0}$ is symmetric as a Frobenius algebra (this is true for all towers in Example \ref{ex-frobenius-ex-structure}), though we do not assume that $E_{n,k}$ is symmetric for $k > 0$ (this is not true for $\{H^\lambda_n\}_{n \geq 0}$ in Example \ref{example-towers-of-algebras}.\ref{ex-cyclo}).

The Frobenius homomorphism $E_{n,k}$ actually gives the co-unit of the adjunction of the functor $\res^{n}_k$ and its right adjoint $\ind^{n}_k$. That is, the natural transformation $\epsilon_{n,k}: \res^{n}_k \ind^{n}_k \rightarrow \UnitModule$ can be translated to an $(A_k,A_k)$-bimodule homomorphism from $_k(A_{n}) \otimes_{A_{n}} (A_{n})_{k} \cong \;_k(A_{n})_k$ to $A_k$. This homomorphism is $E_{n,k}$.

In the language of bimodules, the unit $\eta_{n,k}: \UnitModule \rightarrow \ind^{n}_k \res^{n}_k$ of the adjunction with $\ind^{n}_k$ the right adjoint of $\res^n_k$ defines an $(A_{n},A_{n})$-bimodule homomorphism $H_{n,k}: A_{n} \rightarrow (A_{n})_{k} \otimes_{A_k} \:_{k}(A_{n})$. For $a \in A_{n}$, $H_{n,k}$ sends
\begin{equation*}
a \xmapsto{H_{n,k}} \sum_{b \in B_{n,k}} ab^\vee \otimes b.
\end{equation*}
For any $c \in A_{n}$, $c$ defines an $(A_{n}, A_{n})$-bimodule homomorphism 
\begin{equation*}
T_{c,n,k}: (A_{n})_{k} \otimes_{A_k} \;_{k}(A_{n}) \rightarrow (A_{n})_{k} \otimes_{A_k} \;_{k}(A_{n})
\end{equation*}
via the assignment such that for $a \otimes a' \in (A_{n})_k \otimes_{A_k} \,_k(A_{n})$,
\begin{equation*}
a \otimes a' \xmapsto{T_{c,n,k}} ac \otimes a'.
\end{equation*}
Finally, there is an $(A_{n},A_{n})$-bimodule homomorphism $M_{n,k}: (A_{n})_{k} \otimes_{A_k} \,_{k}(A_{n}) \rightarrow A_{n}$, known as the {\emph{multiplication map}} which maps $a \otimes a' \in (A_{n})_{k} \otimes_{A_k} \,_{k}(A_{n})$ to
\begin{equation*} 
a \otimes a' \xmapsto{M_{n,k}} aa'
\end{equation*}
(in fact this is the co-unit of the adjunction of $\ind^{n}_k$ with $\res^{n}_k$ as its right adjoint). 

We then define $N_{n,k}: A_{n} \rightarrow A_{n}$ so that for $c \in A_{n}$,
\begin{equation*}
c \mapsto (M_{n,k} \circ T_{c,n,k} \circ H_{n,k})(1) = \sum_{b \in B_{n,k}} b^\vee c b.
\end{equation*}

\begin{remark}
When $\{A_n\}_{n \geq 0}$ is the tower of Iwahori-Hecke algebras, the map $N_{n,k}$ is sometimes referred to as the {\emph{relative norm}} \cite{J90}.
\end{remark}


\subsection{The centers of $\{A_n\}_{n \geq 0}$}

We write $Z(A_n)$ for the center of $A_n$ and $Z(A_{n},A_k)$ for the centralizer of $A_k$ in $A_{n}$. That is 
\begin{equation*}
Z(A_{n},A_k) := \{\; a \in A_{n} \;|\: aa' = a'a,\;\; \text{for all }a' \in A_k\}.
\end{equation*}

\begin{proposition} \label{prop-trace-and-conj-send-elements-to-center}
For all $0 \leq k < n$:
\begin{enumerate}
\item \label{part-1-maps-to-center} $E_{n,k}: A_{n} \rightarrow A_k$ maps $Z(A_{n},A_k)$ into $Z(A_k)$. 
\item \label{part-2-maps-to-center} $N_{n,k}: A_{n} \rightarrow A_{n}$ maps $A_{n}$ into $Z(A_{n})$.
\end{enumerate}
\end{proposition}

\begin{proof}
Part \eqref{part-1-maps-to-center} follows from the fact that $E_{n,k}$ is a $(A_{k},A_{k})$-bimodule homomorphism. Let $a \in Z(A_{n},A_k)$. Then for any $a' \in A_k$,
\begin{equation*}
a'E_{n,k}(a) = E_{n,k}(a'a) = E_{n,k}(aa') = E_{n,k}(a)a'.
\end{equation*}
Part \eqref{part-2-maps-to-center} follows from the observation that for any $a \in A_{n}$,
\begin{equation*}
N_{n,k}(a) = \sum_{b \in B_{n,k}} b^\vee a b = (M_{n,k} \circ T_{a,n,k} \circ H_{n,k})(1).
\end{equation*}
Since each map $M_{n,k}$, $T_{a,n,k}$, and $H_{n,k}$ is an $(A_{n},A_{n})$-bimodule homomorphism, then for any $a' \in A_{n}$,
\begin{align*}
a'(M_{n,k} \circ T_{a,n,k} \circ H_{n,k})(1) = (M_{n,k} \circ T_{a,n,k} \circ H_{n,k})(a')\\ = (M_{n,k} \circ T_{a,n,k} \circ H_{n,k})(1)a'.
\end{align*}
\end{proof}

For $0 \leq k < n$, the elements $N_{n,k}(1)$ will play a central role in this paper. We therefore set
\begin{equation*}
C_{n,k} := N_{n,k}(1) = \sum_{b \in B_{n,k}} b^\vee b.
\end{equation*} 
Note that for all $0 \leq k < n$, $C_{n,k}$ belongs to $Z(A_{n})$. In \cite{Kad99}, the element $C_{n,k}$ is referred to as the {\emph{$E_{n,k}$-index}} of the Frobenius extension $A_{n}$ of $A_k$. When $k = 0$ and $E_{n,0}: A_{n} \rightarrow \MB{C}$ is symmetric, $C_{n,0}$ is referred to as the {\emph{central Casimir element}} in \cite{Bro09}. 

\begin{lemma} \cite[Proposition 3.5]{Bro09} \label{lemma-central-casimirs-equal}
When $A_{n}$ is a symmetric Frobenius algebra with respect to $E_{n,0}$ then
\begin{equation*}
C_{n,0} = \sum_{b \in B_{n,0}} b^\vee b = \sum_{b \in B_{n,0}} bb^\vee.
\end{equation*}
\end{lemma}

In general,
\begin{equation} \label{eqn-reverse-casimir}
\sum_{b \in B_{n,k}} b^\vee b \neq \sum_{b \in B_{n,k}} b b^\vee
\end{equation}
(see Example \ref{ex-Casimir}.\ref{ex-Casimir-DAHA} below). Because the element on the right side of \eqref{eqn-reverse-casimir} will appear again later, we denote it by
\begin{equation*}
C^\iota_{n,k} := \sum_{b \in B_{n,k}} b b^\vee.
\end{equation*}

\begin{example} \label{ex-Casimir}
\begin{enumerate}

\item When $A_{n} = \MB{C}[\Sy{n}]$ and $A_k = \MB{C}[\Sy{k}]$ then 
\begin{equation*}
C^\iota_{n,k} = C_{n,k} = \sum_{g \in \Sy{n}/\Sy{k}} g g^{-1} = \Big|\Sy{n}/\Sy{k}\Big| = (n)(n-1)\dots(k+1).
\end{equation*}

\item \label{ex-Casimir-DAHA} Let $\lambda = \omega_i + \omega_j$ for $i,j \in \MB{Z}$ so that the level of $\lambda$ is $d = 2$. Then with the help of Lemma 5.12 in \cite{MS17} a routine calculation shows:
\begin{equation*}
C_{3,2} = 2(x_1 + x_2 + x_3) - 3(i + j) \in H_3^\lambda.
\end{equation*}
On the other hand
\begin{equation*}
C_{3,2}^\iota = 2(s_2s_1x_1s_1s_2 + s_2x_2s_2 + x_3) - 3(i + j)
\end{equation*}
\begin{equation*}
= (6x_3 - 4x_3s_2 - 2s_2s_1s_2) - 3(i +j) \in H_3^\lambda.
\end{equation*}
It is not hard to show that in this case $C^\iota_{3,2}$ is not even in $Z(H_3^\lambda)$.

\end{enumerate}
\end{example}


\section{The representation theory of Frobenius towers} \label{sect-rep-theory}

We begin by setting some notation related to the representation theory of $\{A_n\}_{\geq 0}$ and recalling some fundamental facts. Much of this section is modeled after the representation theory for degenerate cyclotomic Hecke algebras (Example \ref{example-towers-of-algebras}.\ref{ex-cyclo}) in \cite[Part I]{K05}. In this section we assume that $\{A_n\}_{\geq 0}$ is a Frobenius tower. We denote the indexing set of isomorphism classes of simple $A_n$-representations by $\Gamma_n$ (note that this is a finite set since $A_n$ is assumed to be finite-dimensional) and set
\begin{equation*}
\Gamma := \bigcup_{n \geq 0} \Gamma_n.
\end{equation*}
For $\lambda \in \Gamma_n$ let $L^\lambda$ be the simple $A_n$-module corresponding to $\lambda$. Each $L^\lambda$ has a unique projective cover which we denote by $P^\lambda$. The set $\{P^\lambda\}_{\lambda \in \Gamma_n}$ gives a complete list of isomorphism classes of indecomposable projective $A_n$-modules. Recall from Section \ref{sect-Frobenius-extensions} that we write $A_n\Mod$ for the category of finite-dimensional $A_n$-modules. We furthermore set $A_n\Pmod$ to be the category of finite-dimensional projective $A_n$-modules.

We denote by $G_0(A_n\Mod)$ the triangulated Grothendieck group of finite dimensional $A_n$-modules. That is, $G_0(A_n\Mod)$ is the abelian group generated by the set $\{ [M] \;|\: M \in A_n\Mod\}$ subject to the relations such that for $M, N, K \in A_n\Mod$,
\begin{equation*}
[M] + [N] = [K] 
\end{equation*}
if there is a short exact sequence
\begin{equation*}
\zero \rightarrow M \rightarrow K \rightarrow N \rightarrow \zero.
\end{equation*}
$G_0(A_n\Mod)$ is a free abelian group with basis $\{[L^\lambda] \;|\; \lambda \in \Gamma_n\}$. We denote by $K_0(A_n\Pmod)$ the split Grothendieck group of finite dimensional projective $A_n$-modules, so that $K_0(A_n\Pmod)$ is the abelian group generated by $\{ [P] \;|\: P \in A_n\Pmod\}$ with the property that for $P, P', P'' \in A_n\Pmod$,
\begin{equation*}
[P] + [P'] = [P'']
\end{equation*}
if and only if
\begin{equation*}
P \oplus P' \cong P''.
\end{equation*}
$K_0(A_n\Pmod)$ is a free abelian group with basis $\{[P^\lambda] \;|\; \lambda \in \Gamma_n\}$. We set
\begin{equation*}
G_0(A) = \bigoplus_{n \geq 0} G_0(A_n\Mod) \quad\quad \text{and} \quad\quad K_0(A) = \bigoplus_{n \geq 0} K_0(A_n\Pmod).
\end{equation*}

Note that since both $\ind^{n+1}_n$ and $\res^{n+1}_n$ are exact functors that send projectives to projectives, they descend to linear operators on $G_0(A)$ and $K_0(A)$. By abuse of notation we denote both the genuine functors $\ind^{n+1}_n$ and $\res^{n+1}_n$ and their corresponding linear operators on $G_0(A)$ and $K_0(A)$ using the same notation. 

There is a form $\langle \cdot, \cdot \rangle: K_0(A) \times G_0(A) \rightarrow \MB{Z}_{\geq 0}$ such that for $m,n \geq 0$ and $P \in A_m\Pmod$, $L \in A_n\Mod$ then
\begin{equation*}
\langle [P], [L] \rangle = \dim(\Hom_{A_n}(P,L)) \quad\quad \text{if $m = n$}
\end{equation*}
and $\langle [P], [L] \rangle = 0$ otherwise. In fact, reducing to the basis for $K_0(A)$, $\{[P^\lambda]\}_{\lambda \in \Gamma}$, and the basis for $G_0(A)$, $\{[L^\lambda]\}_{\lambda \in \Gamma}$, $\langle \cdot, \cdot \rangle$ can be defined by the property that for $\lambda, \eta \in \Gamma$
\begin{equation} \label{eqn-simple-indecomp-inner-product}
\langle [P^\lambda], [L^\eta] \rangle = \delta_{\lambda,\eta}.
\end{equation}
The equality in \eqref{eqn-simple-indecomp-inner-product} relies on basic properties of projective covers.

Since $\{A_n\}_{n \geq 0}$ is a Frobenius tower, $\ind^{n+1}_n$ and $\res^{n+1}_n$ are biadjoint and hence for any $\mu \in 
\Gamma_n$ and $\lambda \in \Gamma_{n+1}$
\begin{equation} \label{eqn-adjointness-on-inner-prod-1}
\langle \ind^{n+1}_n [P^\mu], [L^\lambda] \rangle = \langle [P^\mu], \res^{n+1}_n [L^\lambda] \rangle
\end{equation}
and
\begin{equation*}
\langle \res^{n+1}_n [P^\lambda], [L^\mu] \rangle = \langle [P^\lambda], \ind^{n+1}_n [L^\mu] \rangle.
\end{equation*}

We define the following elements:
\begin{itemize}

\item  $\kappa(\mu,\lambda)$ for $\mu \in \Gamma_n$ and $\lambda \in \Gamma_{n+1}$ such that for $[P^\mu] \in K_0(A_n\Pmod)$,
\begin{equation} \label{eqn-def-kappa}
\ind^{n+1}_n [P^\mu] = \sum_{\lambda \in \Gamma_{n+1}} \kappa(\mu,\lambda)[P^{\lambda}].
\end{equation}

\item $\kappa^*(\mu,\lambda)$ for $\mu \in \Gamma_n$ and $\lambda \in \Gamma_{n+1}$ such that for $[L^\mu] \in G_0(A_n\Mod)$,
\begin{equation} \label{eqn-def-kappa-star}
\ind^{n+1}_n [L^\mu] = \sum_{\lambda \in \Gamma_{n+1}} \kappa^*(\mu,\lambda)[L^{\lambda}].
\end{equation}

\item  $\bar{\kappa}(\lambda,\mu)$ for $\mu \in \Gamma_n$ and $\lambda \in \Gamma_{n+1}$ such that for $[P^\lambda] \in K_0(A_{n+1})$,
\begin{equation} \label{eqn-def-kappa-bar}
\res^{n+1}_n [P^\lambda] = \sum_{\mu \in \Gamma_{n}} \bar{\kappa}(\lambda,\mu)[P^{\mu}].
\end{equation}

\item $\bar{\kappa}^*(\lambda,\mu)$ for $\mu \in \Gamma_n$ and $\lambda \in \Gamma_{n+1}$ such that for $[L^\lambda] \in G_0(A_{n+1})$,
\begin{equation} \label{eqn-def-kappa-bar-star}
\res^{n+1}_n [L^\lambda] = \sum_{\mu \in \Gamma_{n}} \bar{\kappa}^*(\lambda,\mu)[L^{\mu}].
\end{equation}
\end{itemize}

\begin{remark}
Note that it follows from the definition of $K_0(\cdot)$ that \eqref{eqn-def-kappa} and \eqref{eqn-def-kappa-bar} can be lifted to the level of modules
\begin{equation*}
\ind^{n+1}_n P^\mu \cong \bigoplus_{\lambda \in \Gamma_{n+1}} (P^{\lambda})^{\oplus \kappa(\mu,\lambda)}
\end{equation*}
and
\begin{equation*}
\res^{n+1}_n P^\lambda \cong \bigoplus_{\mu \in \Gamma_{n}} (P^{\mu})^{\oplus \bar{\kappa}(\lambda,\mu)}.
\end{equation*}
Of course, a similar decomposition does not occur in general for $\ind^{n+1}_n L^\mu$ and $\res^{n+1}_n L^\lambda$ when the algebras $\{A_n\}_{n \geq 0}$ are not semisimple. In this case \eqref{eqn-def-kappa-star} and \eqref{eqn-def-kappa-bar-star} can be interpreted at the level of modules via composition series.
\end{remark}

The following proposition is essentially a generalization of Corollary 9.3.2 in \cite{K05}.

\begin{proposition} \label{prop-identifying-kappas}
For all $n \geq 0$ and $\mu \in \Gamma_n$ and $\lambda \in \Gamma_{n+1}$, 
\begin{itemize}
\item $\kappa(\mu,\lambda) = \bar{\kappa}^*(\lambda,\mu)$,
\item $\kappa^*(\mu,\lambda) = \bar{\kappa}(\lambda,\mu)$.
\end{itemize}
\end{proposition}

\begin{proof}
It follows from \eqref{eqn-simple-indecomp-inner-product} and \eqref{eqn-def-kappa} that
\begin{equation*}
\langle \ind^{n+1}_n [P^\mu], [L^\lambda] \rangle = \kappa(\mu,\lambda)
\end{equation*}
while from \eqref{eqn-simple-indecomp-inner-product} and \eqref{eqn-def-kappa-bar-star}
\begin{equation*}
\langle  [P^\mu], \res^{n+1}_n [L^\lambda] \rangle = \bar{\kappa}^*(\lambda,\mu).
\end{equation*}
From \eqref{eqn-adjointness-on-inner-prod-1} we then get that $\kappa(\mu,\lambda) = \bar{\kappa}^*(\lambda,\mu)$. The proof that $\kappa^*(\mu,\lambda) = \bar{\kappa}(\lambda,\mu)$ is analogous.
\end{proof}

Given Proposition \ref{prop-identifying-kappas}, we henceforth use the notation $\kappa(\mu,\lambda)$ and $\kappa^*(\mu,\lambda)$ exclusively. We extend $\kappa$ and $\kappa^*$ to all of $\Gamma \times \Gamma$ by setting 
\begin{equation*}
\kappa(\mu,\lambda) = \kappa^*(\mu,\lambda) = 0
\end{equation*}
when there is no $n \geq 0$ such that $|\mu| = n$ and $|\lambda| = n+1$.

By a classical result, as a left $A_n$-module $A_n$ can be decomposed into a direct sum of indecomposable projective $A_n$-modules so that
\begin{equation} \label{eqn-algebra-decomp}
A_n \cong \bigoplus_{\mu \in \Gamma_n} (P^\mu)^{\dim(L^\mu)}.
\end{equation}
This is a generalization of the well-known decomposition of the left regular representation of a finite group into a direct sum of simple representations. We will use the notation $\{e_{\mu,1}, e_{\mu,2}, \dots, e_{\mu,\dim(L^\mu)}\}$ for a complete set of central, orthogonal idempotents corresponding to the summands $(P^{\mu})^{\oplus \dim(L^\mu)}$.


\subsection{Elements of $Z(A_{n+1},A_n)$ and representations of $A_n$}
\label{subsect-trace-on-centralizers}

Suppose that $x \in Z(A_{n+1},A_n)$, and that $M \in A_{n+1}\Mod$. Then there is an induced action of $x$ on $\res^{n+1}_n(M)$. Indeed, $x \in \End_{A_n}(\res^{n+1}_nM)$ since for any $m \in \res^{n+1}_n(M)$ and $a \in A_{n}$,
\begin{equation*}
a(xm) = x(am).
\end{equation*}
Let $\lambda \in \Gamma_{n+1}$ and recall the decomposition
\begin{equation} \label{eqn-projective-decomposition}
\res^{n+1}_{n} P^\lambda \cong \bigoplus_{\mu \in \Gamma_{n}} (P^{\mu})^{\oplus \kappa^*(\mu,\lambda)}.
\end{equation}
The action of $x$ is stable on each summand in this decomposition since $x$ necessarily commutes with the idempotent in $A_n$ correspond to this summand. It follows that $x \in \End_{A_n}(P^{\mu})$ for each $\mu \in \Gamma_n$, and this action depends on $\lambda \in \Gamma_{n+1}$ where $P^\mu \subseteq \res^{n+1}_n (P^\lambda)$. 

Recall that for $\alpha \in \MB{C}$, the generalized $\alpha$-eigenspace of $x \in \End_{\MB{C}}(M)$ in $M$ is defined to be 
\begin{equation*}
M[\alpha] := \{\; m \in M \;| \; (x - \alpha)^km = 0 \;\text{ for $k >> 0$}\}.
\end{equation*}
Since $(x-a)$ commutes with $A_n$ it follows that as an $A_n$-module
\begin{equation*}
M \cong \bigoplus_{\alpha \in \MB{C}}M[\alpha].
\end{equation*}

In the remainder of this section we will assume that $\{x_n\}_{n \geq 1}$ is a sequence of elements such that for all $n \geq 0$, 
\begin{enumerate}[label=(\alph*)]
\item \label{enum-JM-assumption-1} $x_{n+1} \in Z(A_{n+1},A_n)$,
\item \label{enum-JM-assumption-2} $x_{n+1}$ has constant generalized eigenvalue on $(P^\mu)^{\oplus \kappa^*(\mu,\lambda)} \subseteq \res^{n+1}_n P^\lambda$ for each $\mu \in \Gamma_n$, $\lambda \in \Gamma_{n+1}$.
\end{enumerate}
Under this assumption, to each $\lambda \in \Gamma_{n+1}$ and $\mu \in \Gamma_n$ we can associate $\alpha^{\lambda}_{\mu} \in \MB{C}$ which is the generalized eigenvalue of $x_{n+1}$ on $(P^\mu)^{\oplus \kappa^*(\mu,\lambda)}$ in $\res^{n+1}_{n} P^\lambda$.

\begin{example} \label{example-JM}
The model for the sequence $\{x_n\}_{n \geq 0}$ described above is the collection of Jucys-Murphy elements in $\{\MB{C}[\Sy{n}]\}_{n \geq 0}$ and related algebras. For reference we list some of these elements below:

\begin{enumerate}
\item \label{enum-sym-JM} The classical Jucys-Murphy elements $\{J_n\}_{n \geq 0}$ in $\{\MB{C}[\Sy{n}]\}_{n \geq 0}$, are defined by $J_1 := 0$ and for $n > 1$, 
\begin{equation*}
J_n := (1,n) + (2,n) + \dots + (n-1,n) \in \MB{C}[\Sy{n}].
\end{equation*}
$\{J_n\}_{n \geq 0}$ satisfy assumptions \ref{enum-JM-assumption-1} and \ref{enum-JM-assumption-2}.
\item \label{enum-Serg-JM} The tower of Sergeev algebras $\{\MB{S}_n\}_{n \geq 0}$ has its own version of Jucys-Murphy elements which by abuse of notation we also denote by $\{J_n\}_{n \geq 0}$. These are defined so that $J_1 := 0$ and 
\begin{equation*}
J_n := \sum_{i = 1}^{n-1} (1 + c_ic_n)(i,n).
\end{equation*}
\item \label{enum-WP-JM} Just as Sergeev algebras are just a specific example of Frobenius wreath product algebras $\{F^{\otimes n} \otimes \Sy{n}\}_{n \geq 0}$, the Jucys-Murphy elements of $\{\MB{S}_n\}_{n \geq 0}$ are just a specific case of Jucys-Murphy elements in $\{F^{\otimes n} \otimes \Sy{n}\}_{n \geq 0}$. These are defined so that $J_1 := 0$ and
\begin{equation*}
J_n = \sum^{n-1}_{i = 1} t_{i,n} (1,n) 
\end{equation*}
where 
\begin{equation*}
t_{i,n} = \sum_{b \in B_F} (1^{\otimes i-1} \otimes b \otimes 1^{\otimes n-i-1} \otimes b^\vee).
\end{equation*}
\item \label{enum-cyclo-JM} Finally, the image of $\{x_n\}_{n \geq 1}$ in the quotient algebras $\{H^\lambda_n\}_{n \geq 0}$ satisfy assumptions \ref{enum-JM-assumption-1} and \ref{enum-JM-assumption-2}.
\end{enumerate}

\end{example}

For a representation $N \in A_n\Mod$ write $\Tr_{N}: A_n \rightarrow \MB{C}$ for the usual trace function on images of $A_n$ in $\End_{\MB{C}}(N)$. The normalized trace
\begin{equation*}
\widetilde{\Tr}_N := \frac{\Tr_N}{\dim(N)}
\end{equation*}
is a familiar tool in asymptotic representation theory. In what follows however, we will need to work with the Frobenius homomorphism $E_{n,0}: A_n \rightarrow \MB{C}$ (rather than trace functions) and associate it to a specific projective representation $P$. We do this by pre-composing with multiplication by the corresponding idempotent $e_P$. That is, we consider the map $E_{n,0}(e_P\, \cdot \,): A_n \rightarrow \MB{C}$. In order to arrive at the correct analogue of normalized trace in this setting we rescale $e_P$ to
\begin{equation*}
\widetilde{e}_{P} = \frac{e_P}{\dim(P)}.
\end{equation*}
Note that $\widetilde{e}_P$ is not an idempotent unless $P$ is $1$-dimensional.

Although it is not an element of $A_{n}$, for any $\mu \in \Gamma_n$ we can calculate $\Tr_{P^\mu}(x_{n+1})$ after specifying that $P^\mu \subseteq \res^{n+1}_{n} P^\lambda$ for some $\lambda \in \Gamma_{n+1}$. Then for $k \geq 0$ and $e_{\mu} \in Z(A_n)$ the central idempotent corresponding to some $P^\mu \subseteq \res^{n+1}_{n} P^\lambda$,
\begin{equation*}
\Tr_{A_n}(e_{\mu}x_{n+1}^k) = \Tr_{P^\mu}(x_{n+1}^k) = \dim(P^{\mu})(\alpha^{\lambda}_{\mu})^k
\end{equation*}
and 
\begin{equation*}
\Tr_{A_n}(\widetilde{e}_{\mu}x_{n+1}^k) = (\alpha^{\lambda}_{\mu})^k.
\end{equation*}

Proposition \ref{prop-trace-Trace-C} gives a relationship between the trace on $A_{n+1}$ with respect to left multiplication against $A_{n+1}$, $E_{n+1,0}$, and $C_{n+1,0}$.

\begin{proposition}\cite[Corollary 3.10]{Bro09} \label{prop-trace-Trace-C} 
Let $A_{n+1}$ be a symmetric Frobenius algebra with respect to $E_{n+1,0}: A_{n+1} \rightarrow \MB{C}$ and $\Tr_{A_{n+1}}: A_{n+1} \rightarrow \MB{C}$ be the trace on $A_{n+1}$ with respect to $A_{n+1}$ as a left $A_{n+1}$-module. Then for $a \in A_{n+1}$,
\begin{equation*}
\Tr_{A_{n+1}}(a) = E_{n+1,0}(C_{n+1,0}a).
\end{equation*}
\end{proposition}

\begin{corollary} \cite[Corollary 3.11]{Bro09} \label{cor-characters-of-other-representations}
For $\lambda \in \Gamma_{n+1}$, let $e_{\lambda} \in Z(A_{n+1})$ be a central idempotent corresponding to $P^\lambda$. For $a \in A_{n+1}$, 
\begin{equation*}
\Tr_{P^\lambda}(a) = E_{n+1,0}(C_{n+1,0}e_{\lambda} a)
\end{equation*}
and 
\begin{equation*}
\widetilde{\Tr}_{P^\lambda}(a) = E_{n+1,0}(C_{n+1,0}\widetilde{e_{\lambda}} a).
\end{equation*}
\end{corollary}

\begin{corollary} \label{cor-characters-of-other-representations2}
Let $\lambda \in \Gamma_{n+1}$ and let $e_{\lambda,i}$ and $e_{\lambda,j}$ both be central idempotents of $A_{n+1}$ such that as $A_{n+1}$-modules $e_{\lambda,i}A_{n+1} \cong e_{\lambda,j}A_{n+1} \cong P^\lambda$. Then
\begin{equation*}
E_{n+1,0}(C_{n+1,0}e_{\lambda,i} a) = E_{n+1,0}(C_{n+1,0}e_{\lambda,j} a).
\end{equation*}
\end{corollary}

Proposition \ref{prop-trace-Trace-C} has an analogue for Frobenius towers where $A_0 = \MB{C}$.

\begin{proposition}
Let $\{A_n\}_{n \geq 0}$ be a Frobenius tower with $A_0 = \MB{C}$ and assume $A_{n}$ and $A_{n+1}$ are symmetric Frobenius algebras with respect to $E_{n,0}$ and $E_{n+1,0}$ respectively. Let $\Tr_{A_{n+1}}: A_{n+1} \rightarrow \MB{C}$ be the trace on $A_{n+1}$ with respect to $A_{n+1}$ as a left $A_{n+1}$-module and $\Tr_{A_{n}}: A_{n} \rightarrow \MB{C}$ be the trace on $A_{n}$ with respect to $A_{n}$ as a left $A_{n}$-module. Then for $a \in Z(A_{n+1},A_n)$
\begin{equation*}
\Tr_{A_{n+1}}(a) = \Tr_{A_{n}}(E_{n+1,n}(C_{n+1,n}^\iota a)).
\end{equation*}
\end{proposition}

\begin{proof}
Recall that 
\begin{equation*}
B = \{b^{(n)}b^{(n+1)}  \: | \; b^{(n)} \in B_{n,0}, \;b^{(n+1)} \in B_{n+1,n} \}
\end{equation*}
and 
\begin{equation*}
B^\vee =  \{(b^{(n+1)})^\vee  (b^{(n)})^\vee \;  | \; b^{(n+1)} \in B_{n+1,n},\; b^{(n)} \in B_{n,0} \}
\end{equation*}
are dual bases for $A_{n+1}$ with respect to $E_{n+1,0} = E_{n,0} \circ E_{n+1,n}$. Proposition \ref{prop-trace-Trace-C} implies that for $a \in Z(A_{n+1},A_n)$ we have
\begin{align*}
&\Tr_{A_{n+1}}(a) = E_{n+1,0}(C_{n+1,0} a)) \\
& = \sum_{\substack{b^{(n+1)} \in B_{n+1,n},\\ b^{(n)} \in B_{n,0}}} E_{n+1,0}((b^{(n+1)})^\vee (b^{(n)})^\vee b^{(n)}  b^{(n+1)} a).
\end{align*}
By Lemma \ref{lemma-central-casimirs-equal} this is equal to 
\begin{equation*}
\sum_{\substack{b^{(n+1)} \in B_{n+1,n},\\ b^{(n)} \in B_{n,0}}} E_{n+1,0}(b^{(n)}  b^{(n+1)} (b^{(n+1)})^\vee (b^{(n)})^\vee a).
\end{equation*}
as $E_{n+1,0}$ is assumed to be symmetric. Recalling that $E_{n+1,n}$ is an $(A_n,A_n)$-bimodule homomorphism and $a \in Z(A_{n+1},A_n)$ commutes with $(b)^\vee \in A_n$ we get
\begin{equation*}
\sum_{\substack{b^{(n+1)} \in B_{n+1,n},\\ b^{(n)} \in B_{n,0}}} E_{n,0}(E_{n+1,n}( b^{(n)}  b^{(n+1)} (b^{(n+1)})^\vee  (b^{(n)})^\vee a))
\end{equation*}
\begin{equation} \label{eqn-moving-dual-bases-outside}
= \sum_{ b^{(n)} \in B_{n,0}} E_{n,0}(b^{(n)} E_{n+1,n}(  C_{n+1,n}^\iota a) (b^{(n)})^\vee ).
\end{equation}
Since $E_{n,0}$ is symmetric, \eqref{eqn-moving-dual-bases-outside} is equivalent to
\begin{equation*}
\sum_{ b^{(n)} \in B_{n,0}} E_{n,0}\Big((b^{(n)})^\vee b^{(n)} E_{n+1,n}\Big(  C_{n+1,n}^\iota a\Big) \Big).
\end{equation*}
\begin{equation*}
= E_{n,0}\Big(C_{n,0} E_{n+1,n}\Big(  C_{n+1,n}^\iota a\Big) \Big).
\end{equation*}
Applying Proposition \ref{prop-trace-Trace-C} then gives the desired result.
\end{proof}

The following theorem will allow us to connect the down/up transition functions which will be introduced in Section \ref{sect-coherent-measures} with the centers $\{Z(A_n)\}_{n \geq 0}$.

\begin{theorem} \label{thm-trace-and-moment}
Let $\mu \in \Gamma_n$ and let $e_\mu \in Z(A_n)$ be a central idempotent corresponding to $P^\mu$. Assume that $x_n$ and $x_{n+1}$ satisfy  assumptions \ref{enum-JM-assumption-1} and \ref{enum-JM-assumption-2}. 
\begin{enumerate}

\item For all $k \geq 0$, $E_{n+1,n}(C_{n+1,0}x_{n+1}^k) \in Z(A_n)$ and
\begin{equation*}
\displaystyle E_{n,0}\Big(\widetilde{e}_{\mu} E_{n+1,n}(C_{n+1,0}x_{n+1}^k)\Big) = \sum_{\lambda \in \Gamma_{n+1}} \frac{\kappa^*(\mu,\lambda)\dim(L^\lambda)}{\dim(L^\mu)}(\alpha^\lambda_\mu)^k.
\end{equation*}

\item For all $k \geq 0$, $N_{0}^n(x_n^k) \in Z(A_n)$ and
\begin{equation*}
\displaystyle E_{n,0}\Big(\widetilde{e}_{\mu} N_{0}^n(x_{n}^k)\Big) = \sum_{\eta \in \Gamma_{n-1}} \frac{\kappa^*(\eta,\mu)\dim(P^\eta)}{\dim(P^\mu)}(\alpha^{\mu}_\eta)^k.
\end{equation*}
\end{enumerate}
\end{theorem}

\begin{proof}
\begin{enumerate}
\item $C_{n+1,0} \in Z(A_{n+1},A_n)$ by Proposition \ref{prop-trace-and-conj-send-elements-to-center}.\ref{part-2-maps-to-center}. Since $x_{n+1}$ is also in $Z(A_{n+1},A_n)$ by assumption, then $C_{n+1,0}x_{n+1}^k \in Z(A_{n+1},A_n)$. It then follows from Proposition \ref{prop-trace-and-conj-send-elements-to-center}.\ref{part-1-maps-to-center} that $E_{n+1,n}(C_{n+1,0}x_{n+1}^k) \in Z(A_n)$.

Recall that we write
\begin{equation*}
\res^{n+1}_n P^\lambda = \bigoplus_{\gamma \in \Gamma_{n}} (P^{\gamma})^{\oplus \kappa^*(\gamma,\lambda)}.
\end{equation*}
Hence the trace of the action of $x_{n+1}^k$ on $\res^{n+1}_n P^\lambda$ is precisely,
\begin{equation} \label{eqn-trace-on-restriction}
\Tr_{\res^{n+1}_n P^\lambda}(x_{n+1}^k) = \sum_{\gamma \in \Gamma_{n}} \kappa^*(\gamma,\lambda) \dim(P^\gamma) (\alpha^\lambda_\gamma)^k
\end{equation}
Let $e_{\lambda,\mu}$ be a central idempotent corresponding to $(P^\mu)^{\oplus \kappa^*(\lambda,\mu)}  \subseteq \res^{n+1}_n P^\lambda$. Then 
\begin{equation*}
\Tr_{\res^{n+1}_n P^\lambda}(e_{\lambda,\mu} x_{n+1}^k) = \kappa^*(\lambda,\mu) \dim(P^\mu)(\alpha^\lambda_\mu)^k.
\end{equation*}
Because of the decomposition \eqref{eqn-algebra-decomp} if we let $f_{\mu}$ be the idempotent corresponding to $(P^{\mu})^{\dim(L^\mu)} \subseteq A_n$,
\begin{equation*}
f_{\mu} := \sum_{i = 1}^{\dim(L^\mu)} e_{\mu,i}
\end{equation*}
and if we set
\begin{equation*}
\widehat{f}_\mu := \frac{f_\mu}{\dim(P^\mu)} = \sum_{i = 1}^{\dim(L^\mu)} \widetilde{e}_{\mu,i},
\end{equation*}
then applying \eqref{eqn-trace-on-restriction} to decomposition \eqref{eqn-algebra-decomp} we have
\begin{equation*} 
\Tr_{A_{n+1}}(f_\mu x_{n+1}^k) = \sum_{\lambda \in \Gamma_{n}} \kappa^*(\mu,\lambda)\dim(L^\lambda)\dim(P^\mu)(\alpha^\lambda_\mu)^k
\end{equation*}
and
\begin{equation} \label{eqn-trace-on-reg-rep}
\Tr_{A_{n+1}}(\widehat{f}_\mu x_{n+1}^k) = \sum_{\lambda \in \Gamma_{n}} \kappa^*(\mu,\lambda)\dim(L^\lambda) (\alpha^\lambda_\mu)^k.
\end{equation}
Next by Proposition \ref{prop-trace-Trace-C} 
\begin{equation} \label{eqn-trace-to-conditonal}
\Tr_{A_{n+1}}(\widehat{f}_\mu x_{n+1}^k) = E_{n+1,0}(C_{n+1,0} \widehat{f}_\mu x_{n+1}^k) 
\end{equation}
\begin{equation*}
= E_{n,0}\Big( E_{n+1,n}( C_{n+1,0} \widehat{f}_\mu x_{n+1}^k)\Big)
\end{equation*}
\begin{equation*} 
 = E_{n,0}\Big( \widehat{f}_\mu E_{n+1,n}( C_{n+1,0}  x_{n+1}^k)\Big) 
\end{equation*}
\begin{equation*}
= \dim(L^\mu)E_{n,0}\Big( \widetilde{e}_{\mu}E_{n+1,n}(C_{n+1,0}x_{n+1}^k)\Big)
\end{equation*}
where the third equality follows from the fact that $C_{n+1,0} \in Z(A_{n+1})$ and the last equality follows from Corollary \ref{cor-characters-of-other-representations2}. Combining \eqref{eqn-trace-on-reg-rep} and \eqref{eqn-trace-to-conditonal} then gives the desired result.

\item It follows from Proposition \ref{prop-trace-and-conj-send-elements-to-center}.\ref{part-2-maps-to-center} that $N^n_{0}(x_{n}^k) \in Z(A_n)$.

Because we are assuming that $E_{n,0}$ is symmetric and because $\widetilde{e}_{\mu} \in Z(A_n)$,
\begin{equation*}
E_{n,0}(\widetilde{e}_{\mu} N^n_{0}(x_{n}^k)) = E_{n,0}\Big(\widetilde{e}_{\mu} \sum_{b \in B_{n,0}} b^\vee x_{n}^k b \Big) 
\end{equation*}
\begin{equation*}
= E_{n,0}\Big(\widetilde{e}_{\mu} \sum_{b \in B_{n,0}} bb^\vee x_{n}^k \Big).
\end{equation*}
Then by Lemma \ref{lemma-central-casimirs-equal} and Proposition \ref{prop-trace-Trace-C} we have, 
\begin{equation*}
E_{n,0}\Big(\widetilde{e}_{\mu} \sum_{b \in B_{n,0}} bb^\vee x_{n}^k \Big) = E_{n,0}\Big(\widetilde{e}_{\mu} C_{n,0} x_{n}^k \Big) 
\end{equation*}
\begin{equation*}
= \frac{1}{\dim(P^\mu)}\Tr_{P^\mu}(x_{n}^k).
\end{equation*}
Because $P^\mu$ decomposes as
\begin{equation*}
\res^{n}_{n-1} P^\mu = \bigoplus_{\eta \in \Gamma_{n-1}} (P^{\eta})^{\oplus \kappa^*(\eta,\mu)}
\end{equation*}
and $x_{n}$ has constant generalized eigenvalue $\alpha^\mu_\eta$ on each $P^{\eta}$, then 
\begin{equation*}
\frac{1}{\dim(P^\mu)}\Tr_{P^\mu}(x_{n}^k) = \sum_{\eta \in \Gamma_{n-1}} \frac{\kappa^*(\eta,\mu)\dim(P^\eta)}{\dim(P^\mu)}(\alpha^{\mu}_{\eta})^k.
\end{equation*}

\end{enumerate}

\end{proof}


\section{Two coherent systems on free Frobenius towers} \label{sect-coherent-measures}

Recall that a graded graph is a triple $G = (V,\rho,E)$ where $V$ is a discrete set of vertices, $\rho: V \rightarrow \MB{Z}$ is a rank function, and $E$ is a multiset of edges such that $(x,y) \in E$ for $x,y \in V$ implies that $\rho(y) = \rho(x) +1$ (see \cite{Fom94} for more details). Observe that $\rho$ induces a decomposition of $V$ into graded components $V = \bigcup_{k \in \MB{Z}} V_k$. We begin this section by defining two graded graphs, $G_\Gamma$ and $G_\Gamma^*$ which will be central to this paper.

\begin{definition} \mbox{}
\begin{enumerate}
\item Let $G_\Gamma$ be the graded graph such that:
\begin{enumerate}
\item the vertex set of $G_\Gamma$ is equal to $\Gamma$, and the $n$th graded component is $\Gamma_n$,
\item the vertices $\mu \in \Gamma_n$ and $\lambda \in \Gamma_{n+1}$ only have an edge between them if $\kappa(\mu,\lambda) > 0$.
\end{enumerate}
\item Let $G_\Gamma^*$ be the graded graph such that:
\begin{enumerate}
\item the vertex set of $G_\Gamma$ is equal to $\Gamma$, and the $n$th graded component is $\Gamma_n$,
\item the vertices $\mu \in \Gamma_n$ and $\lambda \in \Gamma_{n+1}$ only have an edge between them if $\kappa^*(\mu,\lambda) > 0$.
\end{enumerate}
\end{enumerate}
\end{definition}

The coherent system formalism was first defined in \cite[Section 1]{BO09} in the context of graded sets. Our usage of this concept in terms of graded graphs is closer to the set-up in \cite[Section 2.2]{Pet09}, but the distinction between all these cases is essentially superficial. 

We begin by reviewing some definitions. For a graded graph $G$ with vertex set $V$ and graded components $V_0, V_1, V_2, \dots$ a {\emph{down transition function}} $p_\downarrow: V \times V \rightarrow [0,1]$ is a function that satisfies the following criteria.
\begin{itemize}
\item For $(\lambda,\mu) \in V \times V$, $p_\downarrow(\lambda,\mu) > 0$ if and only if $\lambda \in V_{n+1}$ and $\mu \in V_n$ for some $n \in \MB{Z}$ and $\lambda$ and $\mu$ are connected by an edge.
\item For fixed $\lambda \in V_{n+1}$,
\begin{equation*}
\sum_{\mu \in V_{n}} p^\downarrow(\lambda,\mu) = 1.
\end{equation*}
\end{itemize}

Let $\{M_n\}_{n \geq 0}$ be a collection of probability measures such that $M_n$ is a probability measure on $V_n$. The measures $\{M_n\}_{n \geq 0}$ are said to be {\emph{coherent}} with respect to down transition function $p_\downarrow$ if for any $\mu \in V_{n}$,
\begin{equation*}
\sum_{\lambda \in V_{n+1}} M_{n+1}(\lambda)p_\downarrow(\lambda,\mu) = M_{n}(\mu).
\end{equation*}
In this paper we assume that for all $n \geq 0$ and $\mu \in V_{n}$, $M_{n}(\mu) > 0$. When this assumption is satisfied, then down transition function $p_\downarrow: V \times V \rightarrow [0,1]$ and coherent measures $\{M_n\}_{n \geq 0}$ together define an {\emph{up-transition function}} $p_\uparrow: V \times V \rightarrow [0,1]$ such that when $\mu \in V_n$ and $\lambda \in V_{n+1}$,
\begin{equation*}
p_\uparrow(\mu,\lambda) = \frac{M_{n+1}(\lambda)}{M_{n}(\mu)} p_\downarrow(\lambda,\mu)
\end{equation*}
and otherwise
\begin{equation*}
p_\uparrow(\mu,\lambda) = 0.
\end{equation*}
By construction the up-transition function $p_\uparrow$ satisfies
\begin{equation*}
\sum_{\lambda \in V_{n+1}} M_n(\mu)p_\uparrow(\mu,\lambda) = M_{n+1}(\lambda).
\end{equation*}
The triple $(p_\downarrow, \{M_n\}_{n \geq 0}, p_\uparrow)$ is called a {\emph{coherent system}}.

Recall that $\{A_n\}_{n \geq 0}$ is a free Frobenius tower and the simple representations of $A_n$ are indexed by the set $\Gamma_n$. We introduce the probability measure $Pl_n: \Gamma_n \rightarrow [0,1]$ defined such that for $\mu \in \Gamma_n$,
\begin{equation*}
Pl_n(\mu) = \frac{\dim(L^\mu)\dim(P^\mu)}{\dim(A_n)}.
\end{equation*}
That this is a probability measure follows from \eqref{eqn-algebra-decomp}. The measures $\{Pl_n\}_{n \geq 0}$ are one generalization of Plancherel measure to the non-semisimple setting. 

For $\mu \in \Gamma_n$ and $\eta \in \Gamma_{n-1}$, set
\begin{equation*}
p_\downarrow(\mu,\eta) = \frac{\kappa(\eta,\mu)\dim(L^\eta)}{\dim(L^\mu)}
\end{equation*}
and 
\begin{equation*}
p_\downarrow^*(\mu,\eta) = \frac{\kappa^*(\eta,\mu)\dim(P^\eta)}{\dim(P^\mu)}.
\end{equation*}

\begin{proposition}
The measures $\{Pl_n\}_{n \geq 0}$ are coherent with respect to $p_\downarrow$ on $G_\Gamma$ and $p_\downarrow^*$ on $G_{\Gamma^*}$.
\end{proposition}

\begin{proof}
For $\mu \in \Gamma_n$ we first observe that since $A_{n+1}$ is assumed to be a free right $A_n$-module, as a right $A_n$-module $A_{n+1}$ must have rank\\ $\dim(A_{n+1})/\dim(A_n)$ and
\begin{equation} \label{eqn-dim-ind-mod-1}
\dim\big( \ind^{n+1}_{n} P^\mu \big) = \frac{\dim(A_{n+1})\dim(P^\mu)}{\dim(A_n)}.
\end{equation}
On the other hand using the notation of \eqref{eqn-def-kappa},
\begin{equation} \label{eqn-dim-ind-mod-2}
\dim\big( \ind^{n+1}_{n} P^\mu \big) = \sum_{\lambda \in \Gamma_{n+1}} \kappa(\mu,\lambda)\dim(P^\lambda).
\end{equation}
Then using the equivalence between \eqref{eqn-dim-ind-mod-1} and \eqref{eqn-dim-ind-mod-2} we have
\begin{align*}
\sum_{\lambda \in \Gamma_{n+1}} M_{n+1}(\lambda)p_\downarrow(\lambda,\mu) = \sum_{\lambda \in \Gamma_{n+1}} \frac{\kappa(\mu,\lambda)\dim(L^\mu)\dim(P^\lambda)}{\dim(A_{n+1})} \\
= \frac{\dim(P^\mu)\dim(L^\mu)}{\dim(A_n)} = M_n(\mu).
\end{align*}
The proof that $\{Pl_n\}_{n \geq 0}$ is coherent with respect to $p_\downarrow^*$ is analogous.
\end{proof}

Observe that it follows by definition that the up-transition function associated to $(p_{\downarrow},\{Pl_n\}_{n \geq 0})$ is defined so that for $\mu \in \Gamma_n$ and $\lambda \in \Gamma_{n+1}$, 
\begin{equation*}
p_\uparrow(\mu,\lambda) = \frac{\dim(A_n)}{\dim(A_{n+1})}\frac{\kappa(\mu,\lambda)\dim(P^\lambda)}{\dim(P^\mu)}
\end{equation*}
while the up-transition function associated to $(p^*_{\downarrow},\{Pl_n\}_{n \geq 0})$ is 
\begin{equation*}
p_\uparrow^*(\mu,\lambda) = \frac{\dim(A_n)}{\dim(A_{n+1})}\frac{\kappa^*(\mu,\lambda)\dim(L^\lambda)}{\dim(L^\mu)}.
\end{equation*}

\begin{remark} Observe that in the coherent system $(p_{\downarrow},\{Pl_n\}_{n \geq 0}, p_{\uparrow})$ on $G_\Gamma$, the dimension of simple representations control the down transition probabilities while the dimension of indecomposable projective representations control the up probabilities. For the coherent system $(p^*_{\downarrow},\{Pl_n\}_{n \geq 0}, p^*_{\uparrow})$ on $G_{\Gamma^*}$ the reverse is true. When all $\{A_n\}_{n \geq 0}$ are semisimple then the distinction between simple $A_n$-modules and indecomposable projective $A_n$-modules disappears and hence $(p_{\downarrow},\{Pl_n\}_{n \geq 0}, p_{\uparrow})$ and $(p^*_{\downarrow},\{Pl_n\}_{n \geq 0}, p^*_{\uparrow})$ (as well as the underlying graphs $G_\Gamma$ and $G_{\Gamma^*}$) are identical.
\end{remark}

Because of assumptions \ref{enum-JM-assumption-1}-\ref{enum-JM-assumption-2}, the elements $\{x_n\}_{n \geq 0}$ define coordinates for elements of $\Gamma$. Specifically, to $\mu \in \Gamma_n$ we associate a finite collection of elements from $\MB{C}$, 
\begin{equation*}
\{\alpha^{\mu}_{\eta} \;|\: \eta \in \Gamma_{n-1}, \, \kappa(\eta,\mu) > 0 \} \bigcup \{\alpha^{\lambda}_{\mu} \;| \; \lambda \in \Gamma_{n+1}, \, \kappa(\mu,\lambda) > 0\}.
\end{equation*}
We get another set of coordinates for $\mu$ by substituting $\kappa^*(\cdot,\cdot)$ for $\kappa(\cdot,\cdot)$ above. In specific examples it is often the case that the sets above take entries in $\MB{Z}$. This is true for Examples \ref{example-JM}.\ref{enum-sym-JM}-\ref{example-JM}.\ref{enum-Serg-JM} and \ref{example-JM}.\ref{enum-cyclo-JM}. When it is the case that the coordinates for $\Gamma$ take values in $\MB{R}$ then to each $\mu \in \Gamma$ we can associate two probability measures $\nu_{\uparrow,\mu}$ and $\nu_{\downarrow,\mu}$ on $\MB{R}$: 
\begin{itemize}
\item The transition measure is
\begin{equation*}
\nu_{\uparrow,\mu} = \sum_{\lambda \in \Gamma_{n+1}} p_\uparrow(\mu,\lambda) \delta_{\alpha^\lambda_\mu}.
\end{equation*}
\item The co-transition measure is
\begin{equation*}
\nu_{\downarrow,\mu} = \sum_{\eta \in \Gamma_{n-1}} p_\downarrow(\mu,\eta) \delta_{\alpha^\mu_\eta}.
\end{equation*}
\end{itemize}
Where for $a \in \MB{R}$, $\delta_a$ is the Dirac delta measure on $\MB{R}$. We can analogously define $\nu^*_{\uparrow,\mu}$ and $\nu^*_{\downarrow,\mu}$ associated to $p_\uparrow^*(\mu,\lambda)$ and $p_\downarrow^*(\mu,\eta)$ respectively.

For $k \geq 0$, we can take the $k$th moments with respect to $\nu_{\uparrow,\mu}$, $\nu_{\downarrow,\mu}$, $\nu^*_{\uparrow,\mu}$, and $\nu^*_{\downarrow,\mu}$. We denote these as $m_{\uparrow,k}(\mu)$, $m_{\downarrow,k}(\mu)$, $m^*_{\uparrow,k}(\mu)$, and $m^*_{\downarrow,k}(\mu)$ respectively. They will be:
\begin{equation} \label{eqn-moment-function-up}
m_{\uparrow,k}(\mu) := \sum_{\lambda \in \Gamma_{n+1}} p_\uparrow(\mu,\lambda) (\alpha^\lambda_\mu)^k,
\end{equation}
\begin{equation} \label{eqn-moment-function-down}
m_{\downarrow,k}(\mu) := \sum_{\eta \in \Gamma_{n-1}} p_\downarrow(\mu,\eta) (\alpha^\mu_\eta)^k,
\end{equation}
\begin{equation} \label{eqn-moment-function-up}
m^*_{\uparrow,k}(\mu) := \sum_{\lambda \in \Gamma_{n+1}} p^*_\uparrow(\mu,\lambda) (\alpha^\lambda_\mu)^k,
\end{equation}
\begin{equation} \label{eqn-moment-function-down}
m^*_{\downarrow,k}(\mu) := \sum_{\eta \in \Gamma_{n-1}} p^*_\downarrow(\mu,\eta) (\alpha^\mu_\eta)^k.
\end{equation}

As suggested by the notation, we may regard $m_{\uparrow,k}$ as a function from $\Gamma \rightarrow \MB{R}$ via the association:
\begin{equation*}
\mu \xmapsto{m_{\uparrow,k}} m_{\uparrow,k}(\mu).
\end{equation*}
We can view $m_{\downarrow,k}$, $m^*_{\uparrow,k}$, and $m^*_{\downarrow,k}$ analogously.

One may ask which pair of coherent systems $(p_{\downarrow},\{Pl_n\}_{n \geq 0}, p_{\uparrow})$ or\\ $(p^*_{\downarrow},\{Pl_n\}_{n \geq 0}, p^*_{\uparrow})$ is more natural from an algebraic perspective. The next theorem shows that at least the moment functions $m^*_{\uparrow,k}$, $m^*_{\downarrow,k}$, associated with coherent system $(p^*_{\downarrow},\{Pl_n\}_{n \geq 0}, p^*_{\uparrow})$ are captured by elements in $\{Z(A_n)\}_{n \geq 0}$. 

\begin{theorem} \label{thm-ind-res-define-moments}
Let $\{A_n\}_{n \geq 0}$ be a free Frobenius tower and assume that $E_{n,0}$ is symmetric for all $n \geq 0$. Let $\{x_n\}_{n \geq 0}$ be elements satisfying assumptions \ref{enum-JM-assumption-1} and \ref{enum-JM-assumption-2}, and let $\mu \in \Gamma$ with $e_\mu$ an idempotent corresponding to $P^\mu$. Then
\begin{enumerate}
\item \label{moment-ident-up} $\displaystyle E_{n,0}\Big(\widetilde{e}_{\mu} E_{n+1,n}(C_{n+1,0}x_{n+1}^k)\Big) = m^*_{\uparrow,k}(\mu).$
\item $\displaystyle E_{n,0}\Big(\widetilde{e}_{\mu,i} N_{0}^n(x_{n}^k)\Big) = \frac{\dim(A_{n})}{\dim(A_{n-1})}m^*_{\downarrow,k}(\mu).$
\end{enumerate}
\end{theorem}

\begin{proof}
This follows from Theorem \ref{thm-trace-and-moment} and definitions \eqref{eqn-moment-function-up}-\eqref{eqn-moment-function-down}.
\end{proof}

Specific examples of the functions $m^*_{\uparrow,k}$ and $m^*_{\downarrow,k}$ have appeared before in the asymptotic representation theory literature. To our knowledge however, this is the first time that they have been written down for a general Frobenius tower, especially when the tower is not assumed to be semisimple. In \cite{Ker93} where $A_n = \MB{C}[\Sy{n}]$ and $\Gamma$ is the set of all Young diagrams, $\{m^*_{k,\uparrow}(\mu)\}_{k \geq 0}$, $\{m^*_{k,\downarrow}(\mu)\}_{k \geq 0}$ were studied in relation to asymptotic properties of the Plancherel measure on symmetric groups. In this setting the coordinates of $\mu$ are Kerov's famous interlacing coordinates and Theorem \ref{thm-ind-res-define-moments}.\ref{moment-ident-up} is an observation by Biane in \cite{B98}. Lassalle proved that $\{m^*_{k,\downarrow}\}_{k \geq 0}$ (properly scaled) and $\{m^*_{k,\uparrow}\}_{k \geq 0}$ are each sets of generators of the shifted symmetric functions \cite[Theorem 6.4]{L09}. 

A similar story was discovered for the tower $\{\MB{S}_n\}_{n \geq 0}$ in \cite{Pet09}, where Petrov's $\mathbf{g}^\uparrow_k$ and $\mathbf{g}^\downarrow_{k+1}$ are identical to our $m^*_{2k,\uparrow}$ and $m^*_{2k,\downarrow}$. It was also shown in \cite{Pet09} that in this case $\{m^*_{2k,\uparrow}\}_{k \geq 0}$ and $\{m^*_{2k,\downarrow}\}_{k \geq 0}$ are each sets of generators for the subalgebra of the symmetric functions generated by odd power sum functions. The connection between $\{Z(\MB{S}_n)\}_{n \geq 0}$ and $\{m^*_{2k,\uparrow}\}_{k \geq 0}$, $\{m^*_{2k,\downarrow}\}_{k \geq 0}$ was observed in \cite[Theorem 2.17]{KOR17}.

More recently, the functions $\{m^*_{k,\uparrow}\}_{k \geq 0}$ and $\{m^*_{k,\downarrow}\}_{k \geq 0}$ have appeared in the context of categorical representation theory. In \cite{KLM16}, the versions associated to the tower $\{ \MB{C}[\Sy{n}]\}_{n \geq 0}$ (i.e. moments of Kerov's transition and co-transition measure) were shown to arise naturally as elements of the center of Khovanov's Heisenberg category \cite[Theorem 5.5]{Kho14}. In \cite[Theorem 5.5]{KOR17}, $m^*_{2k,\downarrow}$ and $m^*_{2k,\uparrow}$ corresponding to $\{\MB{S}_n\}_{n \geq 0}$ appeared as elements of the center of the twisted Heisenberg category of Cautis and Sussan \cite{CLS14}. 

\subsection{Future directions} \label{sect-future-direct}

Below we suggest a number of directions for future research.

\begin{enumerate}

\item {\textbf{New families of limit shapes:}} The framework described above should make it possible to define new families of limit shapes for certain free Frobenius towers. For example, the simple representations of $H^\lambda_n$ are parametrized by proper subsets of multipartitions, which ``grow'' from specific integer locations on $\MB{Z}$ determined by $\lambda$. This growth process should include structure not seen in the classical case of symmetric groups. 

\item {\textbf{Deeper connections to Heisenberg categories:}} As was noted above, the functions $\{m^*_{k,\uparrow}\}_{k \geq 0}$ and $\{m^*_{k,\downarrow}\}_{k \geq 0}$ appear at least twice in categorifications of Heisenberg algebras. This observation is interesting, but what would be more interesting is to understand how far this connection extends and whether Heisenberg categories can shed any new light on problems in asymptotic representation theory. In particular, it seems possible that the diagrammatics of Heisenberg categories might offer a new avenue for studying questions in asymptotic representation theory.

\item {\textbf{Connections to symmetric functions:}} In two of the examples that have been extensively studied ($\{\MB{C}[\Sy{n}]\}_{n \geq 0}$ and $\{\MB{S}_n\}_{n \geq 0}$), the corresponding functions $\{m^*_{k,\uparrow}\}_{k \geq 0}$ and $\{m^*_{k,\downarrow}\}_{k \geq 0}$ can be identified with certain sets of generators for algebras related to symmetric functions. There is evidence that similar results may hold more broadly (for example, when $\{A_n\}_{n \geq 0}$ is a tower of degenerate cyclotomic Hecke algebras). It would be interesting to better understand the properties (especially combinatorial) of the commutative algebras generated by $\{m^*_{k,\uparrow}\}_{k \geq 0}$ and $\{m^*_{k,\downarrow}\}_{k \geq 0}$ respectively.

\end{enumerate}

\bibliographystyle{amsplain}
\bibliography{Frobenius-Refs}

\end{document}